\documentclass[11pt,a4paper]{amsart}
\usepackage{amssymb,amsmath,epsfig,graphics,mathrsfs,enumerate,verbatim}
\usepackage[pagebackref,colorlinks=true,linkcolor=blue,citecolor=blue]{hyperref}
\usepackage{fancyhdr}
\usepackage{hyperref}
\hypersetup{
 colorlinks   = true,
 urlcolor     = blue,
 linkcolor    = blue,
 citecolor   = red ,
 bookmarksopen=true
}

\usepackage{amsmath}
\usepackage{amsfonts}
\usepackage{amssymb}
\usepackage{amsthm}
\usepackage{epsfig,graphics,mathrsfs}
\usepackage{graphicx}
\usepackage{dsfont}

\usepackage[usenames, dvipsnames]{color}

\usepackage{hyperref}

\def \phi {\varphi}

\def \RN {\mathbb{R}^N}
\def \R {\mathbb{R}}

\def \G{\Gamma}

\def \vf{\varphi}

\newcommand{\Rn}{\mathbb R^n}
\newcommand{\Rm}{\mathbb R^m}
\newcommand{\Om}{\Omega}

\newcommand{\p}{\partial}

\newcommand{\la}{\lambda}

\newcommand{\vt}{\vartheta}

\numberwithin{equation}{section}

\newcommand{\beq}{\begin{equation}}
\newcommand{\bea}[1]{\begin{array}{#1} }
\newcommand{\eeq}{ \end{equation}}
\newcommand{\ea}{ \end{array}}

\newcommand{\sa}{\langle}
\newcommand{\da}{\rangle}

\newcommand{\C}{\mathbb{C}}

\newcommand{\Gi}{\mathscr G(X,t)}

\DeclareMathOperator{\sech}{sech}

\newtheorem{theorem}{Theorem}[section]
\newtheorem{lemma}[theorem]{Lemma}
\newtheorem{proposition}[theorem]{Proposition}

\newtheorem{remark}[theorem]{Remark}

\numberwithin{equation}{section}

\begin{document}

\title[On an evolution equation,  etc.]{On an evolution equation in sub-Finsler geometry}




\date{}

\begin{abstract}
We study the gradient flow of an energy with mixed homogeneity which is at the interface of Finsler and sub-Riemannian geometry.
\end{abstract}

\author{Nicola Garofalo}

\address{Dipartimento d'Ingegneria Civile e Ambientale (DICEA)\\ Universit\`a di Padova\\ Via Marzolo, 9 - 35131 Padova,  Italy}
\vskip 0.2in
\email{nicola.garofalo@unipd.it}

\thanks{N. Garofalo is supported in part by a Progetto SID (Investimento Strategico di Dipartimento): ``Aspects of nonlocal operators via fine properties of heat kernels", University of Padova (2022); and by a PRIN (Progetto di Ricerca di Rilevante Interesse Nazionale) (2022): ``Variational and analytical aspects of geometric PDEs". He is also partially supported by a Visiting Professorship at the Arizona State University}

\maketitle

\tableofcontents

\section{Introduction}\label{S:intro}

Singular spaces occupy a prominent position in analysis and geometry. Examples of basic interest are Alexandrov spaces and Finsler manifolds. These latter ambients have received considerable attention over the past few decades as they often occur as scaling limits of crystalline or Riemannian structures, see \cite{Tay} and especially the seminal work \cite{BP}. In this paper we introduce a nonlinear evolution equation which represents the gradient flow of an energy which is at the interface of Finsler and sub-Riemannian geometry. Our main objective is to understand in detail the relevant heat kernel.  
To keep the presentation self-contained we confine ourselves to the model setting of a product of Euclidean spaces, i.e., $\RN= \Rm\times \R^k$, with coordinates $z\in \Rm$, $\sigma\in \R^k$. We assume that on each of the two layers, $\Rm$ and $\R^k$, Minkowski norms $\Phi$ and $\Psi$ have been assigned. By this we mean that
$\Phi^2\in C^{2}(\Rm\setminus\{0\})$, $\Psi^2\in C^{2}(\R^k\setminus\{0\})$, and that these functions are strictly convex and $1$-homogeneous. We respectively denote by $\Phi^0$ and $\Psi^0$ their dual norms defined as in \eqref{dual} below.
The levels sets of the functions $\Phi^0$ and $\Psi^0$ are often referred to as \emph{Wulff shapes}. This situation represents a simplified, yet significant, model for the more general one of a Riemannian manifold in which the tangent space is stratified in layers, each endowed with a different anisotropic structure, and having mixed homogeneities weighted according to the relative position in the stratification. 

We are interested in the $L^2$ gradient flow
\begin{equation}\label{gf}
\frac{\p f}{\p t} = - \frac{\p \mathscr E^{\Phi,\Psi}(f)}{\p f}
\end{equation}
of the following energy 
\begin{equation}\label{Fen}
\mathscr E^{\Phi,\Psi}(f) = \frac 12 \int_{\RN} \left(\Phi(\nabla_z f)^2 + \frac{\Phi^0(z)^{2}}4 \Psi(\nabla_\sigma f)^2\right) dzd\sigma.
\end{equation}
Notable features of \eqref{Fen} are:
\begin{itemize}
\item[(i)] The nonlinear dependence on the degenerate gradient
\begin{equation}\label{dg}
\nabla_X f = (X_1 f,...,X_m f, X_{m+1} f,...,X_N f)
\end{equation}
associated with the $N$ vector fields
\begin{equation}\label{vf}
X_i = \p_{z_i},\ i=1,...,m,\ \ \ X_{m+j} = \frac{\Phi^0(z)}2\  \p_{\sigma_j},\ j=1,...,k.
\end{equation}
(note from \eqref{vf} that for $i=1,...,N$ one has $X_i^\star = - X_i$  in $L^2(\RN)$. We also note that, under the given assumptions on $\Phi$ we have $\Phi^0\in C^1(\Rm\setminus\{0\})$, see \cite[Corollary 1.7.3]{Sc});
\item[(ii)] It degenerates along the manifold $M = \{0\}_z\times \R^k$, but because of the anisotropic nature of the dual norm $\Phi^0(z)$, it does so at different scales along regions of approach to $M$;
\item[(iii)] It is invariant with respect to the  family of mixed dilations in $\RN$
\begin{equation}\label{dil}
\delta_\la(z,\sigma) = (\la z, \la^2 \sigma),\ \ \ \ \ \ \ \la>0;
\end{equation}
\end{itemize}
Returning to \eqref{gf}, a standard argument shows that the relevant evolution PDE attached to \eqref{Fen} is 
\begin{equation}\label{pde}
\frac{\p f}{\p t} = \Delta_\Phi(f) + \frac{\Phi^0(z)^2}{4} \Delta_\Psi(f),
\end{equation}
where we have respectively denoted by  $\Delta_\Phi$ and $\Delta_\Psi$ the Finsler Laplacians in the spaces $\Rm$ and $\R^k$, see \eqref{EL} below. Furthermore, if with $\delta_\la$ as in \eqref{dil} we define 
\begin{equation}\label{bigdel}
\Delta_\la (z,\sigma,t) = (\delta_\la(z,\sigma),\la^2 t),
\end{equation}
then if $f$ solves \eqref{pde}, so does also $f\circ \Delta_\la$, for every $\la>0$.

We emphasise that \eqref{pde} is a nonlinear evolution equation with quadratic growth in the degenerate ``gradient"  \eqref{dg}. By this we mean that there exist  constants $\gamma, \gamma^\star>0$ such that
\begin{equation}\label{qg2}
\gamma\ |\nabla_X f|^2 \le \sa\mathscr A(\nabla_X f),\nabla_X f\da \le \gamma^{\star}\ |\nabla_X f|^2.
\end{equation} 
To explain \eqref{qg2}, note that \eqref{dg} and \eqref{vf} give
\begin{equation}\label{dg2}
\nabla_X f = \begin{pmatrix} \nabla_z f \\ \frac{\Phi^0(z)}2 \nabla_\sigma f\end{pmatrix}.
\end{equation}
Therefore, thanks to the $1$-homogeneity of $\Phi$ and $\Psi$, if we let 
\[
\mathscr A(\nabla_X f) = \begin{pmatrix} \Phi(\nabla_z f) \nabla \Phi(\nabla_z f) \\ \frac{\Phi^0(z)}2 \Psi(\nabla_\sigma f) \nabla \Psi(\eta)\end{pmatrix},
\]
then the PDE \eqref{pde} can be alternatively written as
\begin{equation}\label{pdealt}
\frac{\p f}{\p t} =  \sum_{i=1}^N X_i \mathscr A_i(\nabla_X f),
\end{equation}
where the functions $\mathscr A_i:\RN\to \R$ are the components of the vector field 
\[
\mathscr A(\nabla_X f) = (\mathscr A_1(\nabla_X f),...,\mathscr A_N(\nabla_X f))^T.
\]
 We now note that 
\begin{align}\label{qg}
\sa\mathscr A(\nabla_X f),\nabla_X f\da & = \Phi(\nabla_z f) \sa\nabla \Phi(\nabla_z f),\nabla_z f\da + \frac{\Phi^0(z)^2}4 \Psi(\nabla_\sigma f) \sa\nabla \Psi(\nabla_\sigma f),\nabla_\sigma f\da
\\
& = \Phi(\nabla_z f)^2 + \frac{\Phi^0(z)^2}4 \Psi(\nabla_\sigma f)^2,
\notag
\end{align}
where in the last equality we have used the $1$-homogeneity of $\Phi$ and $\Psi$, which is well-known to be equivalent to  the Euler equations 
\[
\sa\xi,\nabla \Phi(\xi)\da = \Phi(\xi),\ \ \ \ \ \ \sa\eta,\nabla \Psi(\eta)\da = \Psi(\eta).
\]
By the condition \eqref{ell} below, from \eqref{dg2} and \eqref{qg} we finally infer the existence of $\gamma, \gamma^\star>0$ such that \eqref{qg2} hold. Although the present work is not concerned with the local theory of the PDE \eqref{pde},  
we remark that the growth \eqref{qg2} would allow to apply the results in \cite{CCR}, once the relevant volume doubling condition and Poincar\'e inequality are available. For these aspects see the seminal works \cite{Gri}, \cite{SC}, and also the more recent paper \cite{CF} that develops the local theory in Finsler manifolds with Ricci lower bounds.
  
\medskip

To state our first result, we need to introduce some notation. Henceforth, points of $\Rm$ will be denoted by $z, \zeta$, etc., points of $\R^k$ by $\sigma, \tau, \la$, etc. We will use the notation $X = (z,\sigma), Y = (\zeta,\tau)$, etc. for points in the product space $\RN$. Since the notation $\nabla_X$ will no longer appear in this work, there will be no risk of confusion of this notation with the subscript in \eqref{qg2}, \eqref{dg2} and \eqref{qg} above. Also, we will respectively indicate with $\sigma_\Phi$ and $\sigma_\Psi$ the intrinsic measures of the Wulff spheres defined by
\begin{equation}\label{ws}
\sigma_\Phi = \int_{\{\Phi^0(z) = 1\}} \frac{dH^{m-1}(z)}{|\nabla \Phi^0(z)|},\ \ \ \ \ \sigma_\Psi = \int_{\{\Psi^0(\sigma) = 1\}} \frac{dH^{k-1}(\sigma)}{|\nabla \Psi^0(\sigma)|},
\end{equation} 
where $H^{m-1}$ and $H^{k-1}$ denote $(m-1)$-dimensional and $(k-1)$-dimensional Hausdorff measure in $\Rm$ and $\R^k$, respectively. Recall the classical formulas $\sigma_{m-1}= H^{m-1}(\mathbb S^{m-1}) = \frac{2\pi^{\frac m2}}{\G(\frac m2)}$, $\sigma_{k-1} = H^{k-1}(\mathbb S^{k-1}) = \frac{2\pi^{\frac k2}}{\G(\frac k2)}$. It is clear from \eqref{ws} that, in the linear isotropic case $\Phi(z) = |z|$, $\Psi(\sigma) = |\sigma|$, one has $\sigma_\Phi = \sigma_{m-1}$, $\sigma_\Psi = \sigma_{k-1}$. For a number $\nu\in \mathbb C$, we will denote with $J_\nu$ the Bessel function of the first kind and order $\nu$, and indicate by 
\begin{equation}\label{Gnu}
G_\nu(z) = z^{-\nu} J_\nu(z).
\end{equation}
We have the following.

\medskip

\begin{theorem}\label{T:main}
For every $X = (z,\sigma)\in \RN$ and $t>0$ the function
\begin{align}\label{FB}
\Gi & =  \frac{\sigma_{m-1}\sigma_{k-1}}{\sigma_\Phi \sigma_\Psi} \frac{(2\pi)^{-\frac k2} (4\pi)^{-\frac m2} }{t^{\frac m2+k}} \int_0^\infty  \left(\frac{u}{\sinh u}\right)^{\frac m2} e^{-\frac{u}{\tanh u} \frac{\Phi^0(z)^2}{4t}} G_{\frac k2 -1}(\frac{u \Psi^0(\sigma)}t)\ u^{k -1}\ du,
\end{align}
is a solution of the equation \eqref{pde}. Moreover, for every $t>0$ we have
\begin{equation}\label{Gint}
\int_{\RN} \Gi dX = 1.
\end{equation}
\end{theorem}

\begin{remark}\label{R:fs}
We emphasise  that Theorem \ref{T:frompartoel} below proves that, in fact, the function \eqref{FB} is a fundamental solution for \eqref{pde}.  
\end{remark}

\medskip

To state our second result, we mention that in the recent paper \cite{DGGS} the authors have studied in the product space $\RN$ the following degenerate energy with mixed homogeneity 
\begin{equation}\label{Fenp}
\mathscr E_{\alpha,p}(u) = \frac 1p \int_{\RN} \left(\Phi(\nabla_z u)^2 + \frac{\Phi^0(z)^{2\alpha}}4 \Psi(\nabla_\sigma u)^2\right)^{\frac p2} dzd\sigma,\ \ \ \ \ \ \ \ \ \ \ 1<p<\infty.
\end{equation}
It is clear that \eqref{Fen} corresponds to the special case $\alpha = 1$ and $p =2$ of \eqref{Fenp}. One of their main result is an explicit fundamental solution for the Euler-Lagrange equation of \eqref{Fenp}. To formulate the relevant result, consider the following anisotropic Minkowski gauge 
\begin{equation}\label{theta}
\Theta(z,\sigma) = \left(\Phi(z)^{2(\alpha+1)} + 4(\alpha+1)^2 \Psi(\sigma)^2\right)^{\frac{1}{2(\alpha+1)}}.
\end{equation}
In \cite{DGGS} a new Legendre transformation $\Theta^0$ was introduced, 
namely
\begin{equation}\label{theta0}
\Theta^0(z,\sigma)^{\alpha+1}= \underset{\Theta(\xi,\tau) = 1}{\sup}\ \bigg(|\sa z,\xi\da|^{\alpha+1}+ 4(\alpha+1)^2 \sa\sigma, \tau\da\bigg).
\end{equation}
The remarkable feature of \eqref{theta0} is underscored by the following result, established in \cite[Proposition 3.3]{DGGS}: 
\begin{equation}\label{theta00}
\Theta^0(z,\sigma)= \left(\Phi^0(z)^{2(\alpha +1)}+  4(\alpha+1)^2 \Psi^0(\sigma)^2  \right)^{\frac{1}{2(\alpha+1)}},
\end{equation}
where $\Phi^0$ and $\Psi^0$ are the classical dual Minkowski norms defined as in \eqref{dual} below. The reader should observe the perfect symmetry between \eqref{theta} and \eqref{theta00}. Furthermore, in \cite[Theorem 1.2]{DGGS} the authors proved that the function
\begin{equation}\label{wow}
\mathscr E_{\alpha,p}(z,\sigma) = \begin{cases}
C_{\alpha,p}\ \Theta^0(z,\sigma)^{-\frac{Q-p}{p-1}},\ \ \ \ \ \ \ \ p\not= Q,
\\
\\
C_\alpha\ \log \Theta^0(z,\sigma),\ \ \ \ \ \ \ \ \ \ \ \ p = Q,
\end{cases}
\end{equation}
is a fundamental solution, with pole in $(0,0)$,  of the Euler-Lagrange equation of \eqref{Fenp}. The number $Q$ in \eqref{wow} is given by
\begin{equation}\label{Q}
Q = Q_\alpha = m + (\alpha+1) k.
\end{equation}
It is clear from \eqref{wow} that such number plays the role of a dimension. In \eqref{wow}, the positive constants $C_{\alpha,p}$ and $C_\alpha$ are implicitly given, and they involve the Wulff shapes of the gauge $\Theta^0$. However, in order to fully understand the gradient flow of the energy \eqref{Fen}, it is important to have an explicit knowledge of the constants. For this reason, in Section \ref{S:em} we turn to this problem and in Lemma \ref{L:sigmas} we show that, with $\sigma_\Phi$ and $\sigma_\Psi$ as in \eqref{ws} above, then
\begin{equation}\label{sigmas0}
\sigma_{\alpha,p} = \sigma_\Phi \sigma_\Psi\ \frac{B(\frac{m+\alpha p}{2(\alpha+1)},\frac k2)}{2^{k+1} (\alpha+1)^k},
\end{equation}
where we have denoted by $B(x,y)$ Euler beta function, see \eqref{beta}. Since in \cite[Theorem 1.2]{DGGS} it was proved that
\begin{equation}\label{Cippy}
C_{\alpha,p} = \begin{cases} 
\frac{p-1}{Q-p} \left(\sigma_{\alpha,p}\right)^{-1/(p-1)},\ \ \ \ \ p\not= Q,
\\
\\
\sigma_{\alpha,Q}^{-1/(Q-1)},\ \ \ \ \ \ \ \ \ \ p = Q,
\end{cases}
\end{equation}
it is clear that \eqref{sigmas0} provides an expression of $C_{\alpha,p}$.
We explicitly note from \eqref{theta00}, \eqref{wow} and \eqref{Q} that, when $\alpha = 1$ and $p=2$, we have
\begin{equation}\label{theta01}
\Theta^0(z,\sigma)= \left(\Phi^0(z)^{4}+ 16 \Psi^0(\sigma)^2  \right)^{\frac{1}{4}},
\end{equation} 
and 
\begin{equation}\label{Q2}
Q = m + 2k.
\end{equation}
We also mention that the above theorem \eqref{wow} generalised a result first established in \cite{G} for the linear case $\Phi(z) = |z|$, $\Psi(\sigma) = |\sigma|$, when $p =2$.
 
To put our second result in some historical context, consider the standard Gauss-Weierstrass kernel in $\Rn$ (here, we are assuming $n\not= 2$)
\[
G(x,t) = (4\pi t)^{-\frac n2} e^{-\frac{|x|^2}{4t}},\ \ \ \ \ \ t>0.
\]
A well-known manifestation of the fact that $G(x,t)$ is a fundamental solution of the heat equation, and of the Bochner subordination principle, is that the $L^1_{loc}(\RN)$ function defined by   
\[
E(x) \overset{def}{=} \int_0^\infty G(x,t) dt,
\]
provides a fundamental solution of $-\Delta$ with pole in  $0\in \Rn$. Furthermore, an elementary, yet beautiful computation, shows that
\[
E(x) =  \frac{1}{(n-2)\sigma_{n-1}} |x|^{2-n}.
\]
The next theorem  shows that, surprisingly, despite its strongly nonlinear character, the evolution equation \eqref{pde} above displays the same  linear phenomenon as the standard heat equation. Furthermore, Theorem \ref{T:frompartoel} shows the remarkable fact that, had we not known the magic Minkowski gauge $\Theta^0$ in \eqref{theta01}, by running the nonlinear heat flow \eqref{gf} we are forced to discover it!

\medskip

\begin{theorem}\label{T:frompartoel}
Let $\Gi$ be the function defined in \eqref{FB} of Theorem \ref{T:main}. For every $t>0$, we have
\begin{equation}\label{beauty}
\int_0^\infty \Gi dt = \frac{C_{1,2}}{\Theta^0(z,\sigma)^{Q-2}},
\end{equation}
where $\Theta^0$ is as in \eqref{theta01}, $Q$ as in \eqref{Q2}, and $C_{1,2}$ is the constant, identified by \eqref{sigmas0}, and corresponding to the case $\alpha = 1$ and $p=2$  in \eqref{Cippy} above.  
As a consequence of \eqref{wow}, the function defined by the left-hand side of \eqref{beauty} is a fundamental solution of the nonlinear equation 
\begin{equation}\label{pdeti}
\Delta_\Phi(u) + \frac{\Phi^0(z)^2}4 \Delta_\Psi(u) = 0.
\end{equation}
This proves that $\Gi$ is not only a solution to \eqref{pde}, but in fact it is a fundamental solution. 
\end{theorem}

\medskip

 We close this introduction with a short description of the organisation of the present paper. In Section \ref{S:minkia} we collect some basic properties of Minkowski norms, and of the Finsler Laplacian and heat equation. Section \ref{S:em} is devoted to proving Proposition \ref{P:const}. In Section \ref{S:main} we prove Theorem \ref{T:main}. Finally, in Section \ref{S:beau} we establish Theorem  \ref{T:frompartoel}.


\section{Minkowski norms}\label{S:minkia}

Let $M:\Rn\to [0,\infty)$ be a Minkowski norm in $\Rn$. By this we mean that $M^2\in \C^{2}(\Rn\setminus\{0\})$ is a strictly convex function such that $M(\la x) = |\la| M(x)$ for every $x\in \Rn$ and $\la\in \R$. By strict convexity we mean that the Hessian matrix $\frac 12 \nabla^2(M^2)$ in positive definite in $\Rn\setminus\{0\}$. Since all norms in $\Rn$ are equivalent, there exist constants $\beta \ge \alpha>0$ such that
\begin{equation}\label{equinorm}
\alpha |\xi| \le M(\xi) \le \beta |\xi|.
\end{equation} 
We denote by 
\begin{equation}\label{dual}
M^0(x) = \underset{M(\xi)=1}{\sup}\ \sa x,\xi\da,
\end{equation}
its Legendre transform, also known as the dual norm of $M$. The Cauchy-Schwarz inequality trivially holds 
\begin{equation*}
|\sa x,y\da| \le M(x) M^0(y).
\end{equation*}
A basic property of the norms $M$ and $M^0$ is the following, see \cite[Lemma 2.1]{BP}, and also (3.12) in \cite{CS},
\begin{equation}\label{Finabla}
M(\nabla M^0(x)) = M^0(\nabla M(x))  = 1,\ \ \ \ \ \ \ \ x\in \Rn\setminus\{0\}.
\end{equation}
Given a function $u\in C^1(\Rn)$, an elementary, yet useful,  consequence of the homogeneity of $M$ is
\begin{equation}\label{hom}
\sa \nabla M(\nabla u(x)),\nabla u(x)\da = M(\nabla u(x)).
\end{equation}
A less obvious basic fact is the following formula, which can be found in \cite[Lemma 2.2]{BP}.

\begin{lemma}\label{L:BP}
For every $x\in \Rn\setminus\{0\}$, one has
\[
M^0(x) \nabla M(\nabla M^0(x)) = x,\ \ \ \ \ \ M(x) \nabla M^0(\nabla M(x)) = x. 
\]
\end{lemma}

The Euler-Lagrange equation of the energy
\begin{equation}\label{finen}
\mathscr E_{M}(u) = \frac 12 \int M(\nabla u)^2 dx
\end{equation}
is the so-called Finsler Laplacian
\begin{equation}\label{EL}
\Delta_M(u) = \operatorname{div}(M(\nabla u)\nabla M(\nabla u)) = 0.
\end{equation}
It is worth emphasising here that the operator in \eqref{EL} is quasilinear, but not linear, unless of course $M(x) = |x|$. However, the operator $\Delta_M$ is elliptic. In fact, since $M^2$ is homogeneous of degree $2$, Euler formula gives
\[
\sa\nabla(M^2)(\xi),\xi\da = M(\xi)^2,
\]
and from \eqref{equinorm} we thus have for every $\xi\in \Rn$
\begin{equation}\label{ell}
\alpha^2 |\xi|^2 \le \sa\nabla(M^2)(\xi),\xi \da \le \beta^2 |\xi|^2.
\end{equation}
An important property of the dual norm is the following result which follows from \eqref{Finabla} and from Lemma \ref{L:BP}, see \cite{CS} and \cite{FK}.

\begin{proposition}\label{P:fin}
Consider the function 
\[
\psi(x) = M^0(x).
\]
Then if $k\in C^2(\R)$, and $v = k\circ \psi$, one has in $\Rn\setminus\{0\}$
\[
\Delta_M(v) = k''(\psi) + \frac{n-1}\psi k'(\psi).
\]
\end{proposition}

A remarkable consequence of Proposition \ref{P:fin} is that the nonlinear operator $\Delta_M$ acts linearly on functions of the dual norm $M^0$. 
The Finsler heat equation arises from the gradient flow of the energy \eqref{finen}. It is the quasilinear PDE in $\Rn\times \R$
\begin{equation}\label{Mhe}
\p_t u = \Delta_M(u).
\end{equation}
In the framework of Finsler manifolds, an in-depth study of \eqref{Mhe} has been carried in the works \cite{OScpam}, \cite{OSaim}, see also \cite{AIS} for the case of $\Rn$. Using Proposition \ref{P:fin} it is possible to construct the following notable explicit solution of the Finsler heat equation, see \cite[Example 4.3]{OScpam}, and also \cite{AIS}.

\begin{proposition}\label{P:heat}
The function
\[
G(x,t) = t^{-\frac n2} e^{-\frac{M^0(x)^2}{4t}},\ \ \ \ \ \ \ \ \ x\in \Rn, t>0,
\]
solves the heat equation in $\Rn\times (0,\infty)$, i.e., 
\[
\p_t G - \Delta_M(G) = 0.
\]
\end{proposition}

We leave it as an exercise to the reader to verify that the function $G(x,t)$ satisfies the following Proposition \ref{P:LY}, the Finsler counterpart of the extremal case of the famous result in \cite{LY}. A Li-Yau theory in Finsler manifolds was developed in the cited work \cite{OSaim}. In connection with the present work, a version of the theory in \cite{BG} will be presented in a forthcoming article.

\begin{proposition}\label{P:LY}
The following identity is true in $\Rn\times (0,\infty)$
\[
M(\nabla \log G)^2 - \p_t(\log G) = \frac{n}{2t}.
\]
\end{proposition}

Using the coarea formula one can easily show that
\begin{align}\label{c0}
& \int_{\Rn} G(x,t) dx = 2^{n-1} \G(n/2) \sigma_M,
\end{align}
where we have let
\begin{equation}\label{sphere}
\sigma_M = \int_{\{M^0(x) = 1\}} \frac{dH^{n-1}(x)}{|\nabla M^0(x)|}.
\end{equation}
Notice that, on the one hand the coarea formula gives
\begin{equation}\label{wulffball}
\operatorname{Vol}_n(\{x\in \Rn\mid \Phi^0(x)<r\}) = \int_0^r \int_{\{M^0(x) = s\}} \frac{dH^{n-1}(x)}{|\nabla M^0(x)|} ds,
\end{equation}
and therefore
\begin{equation}\label{derwulffball}
\frac{d}{dr} \operatorname{Vol}_n(\{x\in \Rn\mid \Phi^0(x)<r\}) = \int_{\{M^0(x) = r\}} \frac{dH^{n-1}(x)}{|\nabla M^0(x)|}.
\end{equation}
On the other hand, a rescaling gives
\begin{equation}\label{wulffball2}
\operatorname{Vol}_n(\{x\in \Rn\mid \Phi^0(x)<r\}) = \omega_M\ r^n,
\end{equation}
and therefore
\begin{equation}\label{derwulffball2}
\frac{d}{dr} \operatorname{Vol}_n(\{x\in \Rn\mid \Phi^0(x)<r\}) = n \omega_M\ r^{n-1}.
\end{equation}
Equating \eqref{derwulffball} and \eqref{derwulffball2}, we infer
\begin{equation}\label{wulffsphere}
\int_{\{M^0(x) = r\}} \frac{dH^{n-1}(x)}{|\nabla M^0(x)|} = n \omega_M\ r^{n-1},\ \ \ \ \ \ r>0.
\end{equation}
In particular, when $r=1$ we obtain from \eqref{sphere} and \eqref{wulffsphere}
\begin{equation}\label{sm}
\sigma_M = n \omega_M.
\end{equation}

\begin{remark}\label{R:sm}
The identity \eqref{sm} can also be directly obtained by the well-known formula of Minkowski for the volume, which gives
\begin{equation}\label{minkio}
\omega_M = \frac 1n \int_{\{M^0(x) = 1\}} \sa x,\nu\da dH^{n-1}(x).
\end{equation}
If we now use the first identity in Lemma \ref{L:BP}, the fact that $\nu = \frac{\nabla M^0}{|\nabla M^0|}$, and Euler formula $\sa\nabla M(p),p\da = M(p)$, applied with $p = \nabla M^0(x)$,  we find on $\{M^0(x) = 1\}$
\[
\sa x,\nu\da = M^0(x) \sa\nabla M(\nabla M^0(x)),\frac{\nabla M^0(x)}{|\nabla M^0(x)|}\da = \frac{M(\nabla M^0(x))}{|\nabla M^0|} = \frac{1}{|\nabla M^0|},
\] 
where in the last equality we have used \eqref{Finabla}. Substituting in \eqref{minkio}, and using \eqref{sphere}, we obtain \eqref{sm}.
\end{remark}

We will need the following useful observation.

\begin{lemma}\label{L:radial}
Suppose $f(x) = f^\star(|x|)$, for some measurable function $f^\star:[0,\infty)\to \R$, and consider the function $F(x) = f^\star(M^0(x))$. Then
\[
\int_{\Rn} F(x) dx = \frac{\sigma_M}{\sigma_{n-1}} \int_{\Rn} f(x) dx,
\]
where as customary $\sigma_{n-1} = \frac{2\pi^{\frac n2}}{\G(n/2)}$.
\end{lemma}

\begin{proof}
The coarea formula and \eqref{wulffsphere} give
\begin{align*}
\int_{\Rn} F(x) dx & = \int_0^\infty f^\star(r) \int_{\{M^0(x) = r\}} \frac{dH^{n-1}(x)}{|\nabla M^0(x)|} dr
\\
& = \sigma_M \int_0^\infty f^\star(r) r^{n-1} dr = \frac{\sigma_M}{\sigma_{n-1}} \int_{\Rn} f(x) dx.
\end{align*}

\end{proof}


\section{Equalities matter}\label{S:em}

\medskip

This section is devoted to the explicit computation of the constants $C_{\alpha,p}$ and $C_\alpha$ in the above cited \eqref{wow} from \cite[Theorem 1.2]{DGGS}. The main result of the section is Proposition \ref{P:const}. This result plays a critical role in recognising that the function $\Gi$ introduced in \eqref{FB} of Theorem \ref{T:main} above, is in fact, a fundamental solution of the PDE \eqref{pde}.  Throughout this section, $\RN = \Rm\times \R^k$, and $\Phi$ and $\Psi$ will be strictly-convex Minkowski norms in $\Rm$ and $\R^k$, respectively. We recall from \cite{DGGS} that, with $\Theta^0(z,\sigma)$ given by \eqref{theta00} above, the constants in \eqref{wow} are prescribed by the  formulas \eqref{Cippy}, 
where, with $Q$ as in \eqref{Q} above, we have let
\begin{equation}\label{sigmas}
\sigma_{\alpha,p} = Q \int_{\{\Theta^0(z,\sigma) <1\}} \left(\frac{\Phi^0(z)}{\Theta^0(z,\sigma)}\right)^{\alpha p} dz d\sigma.
\end{equation}
If we now consider the Wulff ball of the anisotropic gauge 
\[
\mathscr K = \{(z,\sigma)\in \RN\mid \Theta^0(z,\sigma)<1\},
\]
and set
\[
\mathscr K^\star = \{(r,s)\in [0,\infty)\times [0,\infty)\mid r^{2(\alpha+1)} + 4(\alpha+1)^2 s^2 <1\},
\]
then it should be clear to the reader that, if we define
\begin{equation}\label{fpiccola}
f^\star(r,s) = \frac{r^{\alpha p}}{\left(r^{2(\alpha+1)} + 4(\alpha+1)^2 s^2\right)^{\frac{\alpha p}{2(\alpha+1)}}}\ \mathbf 1_{\mathscr K^\star}(r,s),
\end{equation}
where we have denoted by $\mathbf 1_{\mathscr K^\star}$ the indicator function of the set $\mathscr K^\star$, then we have
\begin{equation}\label{fgrande}
F(z,\sigma) = \left(\frac{\Phi^0(z)}{\Theta^0(z,\sigma)}\right)^{\alpha p} \mathbf 1_{\mathscr K}(z,\sigma) = f^\star(\Phi^0(z),\Psi^0(\sigma)).
\end{equation}
 
We now state a generalisation of Lemma \ref{L:radial}, whose proof we leave to the reader.

\begin{lemma}\label{L:doubleradial}
Let $f^\star:[0,\infty)\times [0,\infty)\to \overline \R$ be a measurable function, and let 
\[
F(z,\sigma) = f^\star(\Phi^0(z),\Psi^0(\sigma)),\ \ \ \ \ \ \ \ f(z,\sigma) = f^\star(|z|,|\sigma|).
\]
Then
\begin{equation}\label{drad}
\int_{\RN} F(z,\sigma) dz d\sigma = \frac{\sigma_\Phi \sigma_\Psi}{\sigma_{m-1}\sigma_{k-1}} \int_{\RN} f(z,\sigma) dz d\sigma.
\end{equation}
Furthermore, one has
\begin{equation}\label{drad2}
\int_{\RN} f(z,\sigma) dz d\sigma = \sigma_{m-1}\sigma_{k-1} \int_0^\infty \int_0^\infty f^\star(r,s) r^{m-1} s^{k-1} dr ds.
\end{equation}
Therefore, by combining \eqref{drad} and \eqref{drad2}, we find
\begin{equation}\label{drad3}
\int_{\RN} F(z,\sigma) dz d\sigma = \sigma_\Phi \sigma_\Psi \int_0^\infty \int_0^\infty f^\star(r,s) r^{m-1} s^{k-1} dr ds.
\end{equation}
\end{lemma}

If we now apply \eqref{drad3}, keeping \eqref{sigmas}, \eqref{fpiccola} and \eqref{fgrande} in mind, we reach the following conclusion.
\begin{lemma}\label{L:fgo}
For every $\alpha, m, k, >0$, with $Q$ given as in \eqref{Q}, we have 
\[
\sigma_{\alpha,p} = \sigma_\Phi \sigma_\Psi\ Q \int_{\mathscr K^\star} \frac{r^{\alpha p}}{\left(r^{2(\alpha+1)} + 4(\alpha+1)^2 s^2\right)^{\frac{\alpha p}{2(\alpha+1)}}}  r^{m-1} s^{k-1} dr ds.
\]
\end{lemma}

In the next lemma, we compute in closed form the integral in the right-hand side of the equation in Lemma \ref{L:fgo}. Such result will be crucial in the proof of Proposition \ref{P:const} below. In its statement we denote by 
\begin{equation}\label{beta}
B(x,y) = 2 \int_0^{\frac{\pi}{2}} \big(\cos \vt\big)^{2x-1}
\big(\sin \vt\big)^{2y-1} d\vt
\end{equation}
Euler beta function. For the reader's convenience, we recall its well-known property
\begin{equation}\label{bg}
B(x,y) = \frac{\Gamma(x) \Gamma(y)}{\Gamma(x+y)}.
\end{equation}

\begin{lemma}\label{L:crucial}
For any $p>1$, and $m, k>0$, one has
\[
\int_{\mathscr K^\star} \frac{r^{\alpha p}}{\left(r^{2(\alpha+1)} + 4(\alpha+1)^2 s^2\right)^{\frac{\alpha p}{2(\alpha+1)}}} r^{m-1} s^{k-1} dr ds=   \frac{B(\frac{m+\alpha p}{2(\alpha+1)},\frac k2)}{2^{k+1} (\alpha+1)^k (m+(\alpha+1) k)}.
\]
\end{lemma}

\begin{proof}

Consider the open quadrant $\Om = (0,\infty)\times (0,\infty)$ in $(r,s)$-plane. Define a mapping from the open strip $U = (0,\infty)\times (0,\frac \pi{2})$ of the $(\rho,\vt)$-plane onto $\Om$ by the equations
\[
r = \rho (\cos \vt)^{\frac{1}{\alpha+1}},\ \ \ \ \ \ \ s = \frac{\rho^{\alpha + 1}}{2(\alpha+1)} \sin \vt.
\]
The Jacobian determinant of such mapping is
\[
J(\rho,\vt) = \frac{\rho^{\alpha + 1}}{2(\alpha+1)}  (\cos \vt)^{-\frac{\alpha}{\alpha+1}},
\]
therefore $(\rho,\vt)\to (r,s)$ is a diffeomorphism of $U$ onto $\Om$. Jacobi formula for the change of variable and the defining equation \eqref{beta} thus give
\begin{align*}
& \int_{\mathscr K^\star} \frac{r^{\alpha p}}{\left(r^{2(\alpha+1)} + 4(\alpha+1)^2 s^2\right)^{\frac{\alpha p}{2(\alpha+1)}}} r^{m-1} s^{k-1} dr ds
\\
& = \frac{1}{2^k(\alpha+1)^k} \int_0^1 \rho^{m+(\alpha+1)k - 1} d\rho \int_0^{\frac {\pi}2} (\cos \vt)^{\frac{m+\alpha p}{\alpha+1} - 1} (\sin \vt)^{k-1} d\vt
\\
& = \frac{1}{2^{k+1}(\alpha+1)^k} \frac{B(x,y)}{m+(\alpha +1)k},
\end{align*} 
where
\[
2x -1 = \frac{m+\alpha p}{\alpha+1} - 1,\ \ \ \ \ \ 2y - 1 = k - 1.
\]
This completes the proof.

\end{proof}

Combining Lemmas \ref{L:fgo} and \ref{L:crucial}, we obtain the following result.

\begin{lemma}\label{L:sigmas}
For every $m, k \in \mathbb N$, $\alpha>0$ and $p>1$, one has
\begin{equation}\label{sigmas2}
\sigma_{\alpha,p} = \sigma_\Phi \sigma_\Psi\ \frac{B(\frac{m+\alpha p}{2(\alpha+1)},\frac k2)}{2^{k+1} (\alpha+1)^k}.
\end{equation}
\end{lemma}

In connection with Theorem \ref{T:frompartoel}, we are particularly interested in the case $\alpha = 1$
and $p = 2$ of Lemma \ref{L:sigmas}. Our objective is establish the following important fact.

\medskip

\begin{proposition}\label{P:const}
When $\alpha = 1$ and $p = 2$, the constant in \eqref{Cippy} admits the following alternate representation
\begin{equation}\label{C}
C_{1,2}  = \frac{1}{(Q-2)\sigma_{1,2}}  = \frac{\sigma_{m-1}\sigma_{k-1}}{\sigma_\Phi\sigma_\Psi} \frac{2^{\frac m2 +2k-4} \G(\frac m4) \G(\frac 12(\frac m2 + k -1))}{\pi^{\frac{m+k+1}2}}.
\end{equation}
As a consequence of \eqref{C}, and of the above cited \eqref{wow} from \cite{DGGS}, the function defined by the integral in the left-hand side of \eqref{beauty} is the fundamental solution of the nonlinear PDE \eqref{pdeti}.
\end{proposition}

\begin{proof}
It is clear that, to prove the proposition, it suffices to show that
\begin{equation}\label{fck}
\sigma_{1,2} =  \frac{\sigma_\Phi\sigma_\Psi}{\sigma_{m-1}\sigma_{k-1}} \frac{\pi^{\frac{m+k+1}2}}{(m+2k-2) 2^{\frac m2 +2k-4} \G(\frac m4) \G(\frac 12(\frac m2 + k -1))},
\end{equation}
where we have used \eqref{Q2} to write $Q-2 = m+2k -2$. Since according to \eqref{sigmas2} we have
\[
\sigma_{1,2} = \sigma_\Phi \sigma_\Psi\ \frac{B(\frac{m+2}{4},\frac k2)}{2^{2k+1}},
\]
keeping \eqref{bg} in mind, we infer that, in order to establish \eqref{fck}, we need to prove the following identity
\begin{equation}\label{fck2}
\frac{\G(\frac{m}{4} + \frac 12) \G(\frac k2)}{\G(\frac{m+2k + 2}{4})} = \frac{1}{\sigma_{m-1}\sigma_{k-1}} \frac{\pi^{\frac{m+k+1}2}}{(m+2k-2) 2^{\frac m2-5} \G(\frac m4) \G(\frac 12(\frac m2 + k -1))}.
\end{equation}
Since $\sigma_{m-1} = \frac{2\pi^{\frac m2}}{\G(\frac m2)}$, $\sigma_{k-1} = \frac{2\pi^{\frac k2}}{\G(\frac k2)}$, we see that \eqref{fck2} is equivalent to proving that 
\begin{equation}\label{fck3}
\frac{\G(\frac 12(\frac{m}{2} + 1))}{\G(\frac 12(\frac{m}{2}+k+1))} =  \frac{\sqrt \pi\  \G(\frac m2)}{(m+2k-2) 2^{\frac m2-3} \G(\frac m4) \G(\frac 12(\frac m2 + k -1))}.
\end{equation}
At this point we use twice the Legendre duplication formula (see e.g. \cite[(1.2.3) on p.3]{Le})
\begin{equation}\label{prod}
2^{2x-1} \G(x) \G(x+\frac 12) = \sqrt \pi \G(2x).
\end{equation}
The first time, we take $x = \frac 12(\frac{m}{2} + 1)$, for which we have 
\[
2x -1 = \frac m2,\ \ \ \ \ \ \  x + \frac 12 = \frac m4 + 1,
\]
obtaining
\begin{equation}\label{num}
\G(\frac 12(\frac{m}{2} + 1)) = \frac{\sqrt \pi\ \G(\frac m2)}{2^{\frac m2 - 1} \G(\frac m4)}.
\end{equation}
For the reader's convenience we mention that we have repeatedly used the well-known formula $\G(\nu + 1) = \nu \G(\nu)$. 
The second time, we take $x = \frac 12(\frac{m}{2}+k+1)$, which gives
\[
2x -1 = \frac m2 + k,\ \ \ \ \ \ \ x + \frac 12 = \frac 12(\frac m2+k) + 1.
\]
We thus find
\begin{equation}\label{den}
\frac{1}{\G(\frac 12(\frac{m}{2}+k+1))} = \frac{2^{\frac m2 + k -1} \G(\frac 12(\frac m2 + k))}{\sqrt \pi\ \G(\frac m2 + k)}.
\end{equation}
Combining \eqref{num} with \eqref{den}, we see that now \eqref{fck3} is reduced to verifying whether the identity
\begin{equation}\label{fck4}
2^{\frac m2 + k-3}\ \G(\frac 12(\frac m2 + k -1)) \G(\frac 12(\frac{m}{2} + k)) =  \frac{\sqrt \pi\ \G(\frac{m}{2}+k)}{(m+2k-2)}
\end{equation}
is true or not. This can be accomplished by one last application of \eqref{prod}, this time with $x = \frac 12(\frac m2 + k -1)$, for which
\[
2x -1 =  \frac m2 + k -2,\ \ \ \ \ \ \ x + \frac 12 = \frac 12(\frac{m}{2} + k).
\]
We thus find that \eqref{fck4} becomes equivalent to verifying that the following identity holds
\[
\frac{\G(\frac m2 + k -1)}{2} =  \frac{\G(\frac m2 + k)}{m+2k-2}.
\]
This is obviously true as a consequence of $\G(\nu+1) = \nu \G(\nu)$. We have thus completed the proof of  Proposition \ref{P:const}.

\end{proof}

\medskip

 
\section{Construction of the heat kernel}\label{S:main}

\medskip

In this section we prove Theorem \ref{T:main}. We begin by recalling that, with $G_\nu$ defined by \eqref{Gnu} above, 
if 
$h^\star:\R^+ \to \C$ is a measurable function, its Hankel
transform of order $\nu$ is defined by
\begin{equation}\label{BT}
\mathcal H_\nu(h^\star)(s) = (2\pi)^{\nu+1}  \int_0^\infty h^\star(u) G_\nu(2\pi u s) u^{2\nu+1} du,
\end{equation}
see p. 18 in \cite{MS}, and also \cite[Sec. 22]{Gft}. Suppose now that $h(\la) = h^\star(|\la|)$, with $\la\in \R^k$ is a spherically symmetric function in $\R^k$. Its Fourier transform is given by the well-known Bochner formula 
\begin{equation}\label{boch}
\hat h(\sigma) = \frac{2\pi}{|\sigma|^{\frac k2 -1}} \int_0^\infty h^\star(u) J_{\frac k2 -1}(2\pi |\sigma| u) u^{\frac k2} du.
\end{equation}
Comparing \eqref{boch} with \eqref{BT}, we see that the Hankel representation of $\hat h$ is given by
\begin{equation}\label{hankelFT}
\hat h(\sigma) = \mathcal H_{\frac{k-2}{2}}(h^*)(|\sigma|)
 = (2\pi)^{\frac k2} \int_0^\infty  h^\star(u) G_{\frac k2-1}(2\pi  u |\sigma|) u^{k-1} du.
\end{equation}

We next recall the following result, see \cite[Theorem 3.4]{GT}.
\begin{theorem}\label{T:bg}
Given the \emph{PDE} in $\RN\times \R$, 
\begin{equation}\label{bgE}
\p_t f = \Delta_z f + \frac{|z|^2}{4} \Delta_\sigma f,
\end{equation}
its heat kernel with pole at $X' = (z',\sigma')\in \RN$ of is given by
\begin{align*}
H(X,X',t) & =  \frac{2^k}{(4\pi t)^{\frac{m}2 +k}} \int_{\R^k} e^{- \frac it \langle \la,\sigma'-\sigma\rangle}   \left(\frac{|\la|}{\sinh |\la|}\right)^{\frac m2} 
   e^{-\frac{|\la|}{4t \tanh |\la|} ((|w|^2 +|w'|^2) - 2 \langle w,w'\rangle \sech |\la|)} d\la.
\end{align*}
\end{theorem} 
The fact that $H(X,X',t)$ is the heat kernel of \eqref{bgE} means that, for every fixed $X'\in \RN$, the function $(X,t)\to H(X,X',t)$ solves \eqref{bgE}, and that furthermore one has
\[
\int_{\RN} H(X,X',t) \vf(X') dX'\ \underset{t\to 0^+}{\longrightarrow}\ \vf(X).
\]
Henceforth, when $X' = 0\in \RN$, we use the abbreviated notation $H(X,t) = H(X,0,t)$, so that
\begin{equation}\label{GG}
H(X,t)  = \frac{2^k}{(4\pi t)^{\frac{m}2 +k}} \int_{\R^k} e^{-\frac{i}{t} \sa \sigma,\la \da}   \left(\frac{|\la|}{\sinh |\la|}\right)^{\frac m2} e^{-\frac{|\la|}{\tanh |\la|} \frac{|z|^2}{4t}} d\la,\ \ \ \ \ \ t>0.
\end{equation}
It is clear from \eqref{hankelFT} and \eqref{GG} that, if we consider in $\R^k$ the rapidly decreasing spherically symmetric function 
\begin{equation}\label{h}
\la\ \longrightarrow\ h(\la) = h^\star(|\la|) = \left(\frac{|\la|}{\sinh |\la|}\right)^{\frac m2} e^{-\frac{|\la|}{\tanh |\la|} \frac{|z|^2}{4t}},
\end{equation}
then
\begin{equation}\label{H}
H(X,t) = \frac{2^k}{(4\pi t)^{\frac{m}2 +k}}\ \hat h(\frac{\sigma}{2\pi t}),
\end{equation} 
and therefore we have the following.

\begin{lemma}\label{L:hankel}
The function in \eqref{GG} can be written as
\begin{align}\label{FB2}
H(X,t) & =  \frac{(2\pi)^{-\frac k2} (4\pi)^{-\frac m2} }{t^{\frac m2+k}}  \int_0^\infty  \left(\frac{u}{\sinh u}\right)^{\frac m2} e^{-\frac{u}{\tanh u} \frac{|z|^2}{4t}} G_{\frac k2 -1}(\frac{u |\sigma|}t)\ u^{k -1}\ du.
\end{align}
\end{lemma}

We reiterate that the function $H(X,t)$ solves in $\RN\times (0,\infty)$ the PDE \eqref{bgE}, i.e.,
\begin{equation}\label{pdeagain}
H_t = \Delta_z H + \frac{|z|^2}4 \Delta_\sigma H.
\end{equation}

We can now provide the 

\medskip

\begin{proof}[Proof of Theorem \ref{T:main}]
It is clear that $H(X,t)$ is spherically symmetric both in $z\in \Rm$ and $\sigma \in \R^k$. In fact, from \eqref{FB2} in Lemma \ref{L:hankel}, we can write it in the form 
\begin{equation}\label{HF}
H(X,t) = F(|z|,|\sigma|,t),
\end{equation}
where 
\begin{equation}\label{Frst}
F(r,s,t) = \frac{(2\pi)^{-\frac k2} (4\pi)^{-\frac m2} }{t^{\frac{m}2 +k}} \int_0^\infty \left(\frac{u}{\sinh u}\right)^{\frac m2} e^{-\frac{u}{\tanh u} \frac{r^2}{4t}}\ G_{\frac k2 -1}(\frac{u s}t)\ u^{k -1}\ du.
\end{equation}
From the fact that $H(X,t)$ solves the PDE \eqref{pdeagain} in $\RN\times (0,\infty)$, we infer that the function \eqref{F} satisfies the following PDE in the variables $(r,s,t)\in \R^+\times\R^+\times(0,\infty)$
\begin{equation}\label{bgrst}
F_t = \left\{F_{rr}+ \frac{m-1}r F_r + \frac{r^2}4 \left(F_{ss} + \frac{k-1}s F_s\right)\right\}.
\end{equation}
Since, on the other hand, it is obvious from \eqref{FB} and \eqref{Frst} that
\begin{equation}\label{GiF}
\Gi = \frac{\sigma_{m-1}\sigma_{k-1}}{\sigma_\Phi \sigma_\Psi} F(\Phi^0(z),\Psi^0(\sigma),t),
\end{equation}
from the chain rule and a double application of Proposition \ref{P:fin}, we infer
\begin{align}\label{Delphi}
& \left\{\Delta_\Phi(\mathscr G) + \frac{\Phi^0(z)^2}{4} \Delta_\Psi(\mathscr G)\right\}(X,t)
\\
& = \frac{\sigma_{m-1}\sigma_{k-1}}{\sigma_\Phi \sigma_\Psi} \left\{F_{rr} + \frac{m-1}{r} F_r + \frac{r^2}4 \left(F_{ss} + \frac{k-1}{s} F_s\right)\right\}(\Phi^0(z),\Psi^0(\sigma),t)
\notag
\\
& = \frac{\sigma_{m-1}\sigma_{k-1}}{\sigma_\Phi \sigma_\Psi} F_t(\Phi^0(z),\Psi^0(\sigma),t) = \mathscr G_t(X,t),
\notag
\end{align}
where in the second to the last equality we have used \eqref{bgrst}.
The equation \eqref{Delphi} shows that the function $\Gi$ solves the nonlinear PDE \eqref{pde}.

We next prove \eqref{Gint}. We have from \eqref{FB} and Fubini's theorem
\begin{align*}
\int_{\RN} \Gi dX & =  \frac{\sigma_{m-1}\sigma_{k-1}}{\sigma_\Phi \sigma_\Psi} \frac{(2\pi)^{-\frac k2} (4\pi)^{-\frac m2} }{t^{\frac m2+k}}  \int_{\R^k} \int_0^\infty \left(\frac{u}{\sinh u}\right)^{\frac m2} G_{\frac k2 -1}(\frac{u \Psi^0(\sigma)}t)\ u^{k -1}
\\
& \times  \int_{\Rm}  e^{-\frac{u}{\tanh u} \frac{\Phi^0(z)^2}{4t}} dz\ du\ d\sigma.
\end{align*}
By the homogeneity of $\Phi^0$ and a change of variable, we have
\begin{align*}
& \int_{\Rm}  e^{-\frac{u}{\tanh u} \frac{\Phi^0(z)^2}{4t}} dz = \int_{\Rm}  e^{-\Phi^0(\sqrt{\frac{u}{\tanh u} \frac{1}{4t}} z)^2} dz
= \left(\frac{\tanh u}{u}\right)^{\frac m2} 2^m t^{\frac m2}\int_{\Rm}  e^{-\Phi^0(z)^2} dz.
\end{align*}
Now notice that Lemma \ref{L:radial} gives
\begin{align*}
&  \int_{\Rm}  e^{-\Phi^0(z)^2} dz = \frac{\sigma_\Phi}{\sigma_{m-1}} \pi^{\frac m2}.
\end{align*}
Substitution in the above equation, thus gives
\begin{align*}
& \int_{\RN} \Gi dX = \frac{\sigma_{k-1}}{\sigma_\Psi}
  \frac{(2\pi)^{-\frac k2}}{t^{k}} \int_{\R^k} \int_0^\infty \left(\frac{1}{\cosh u}\right)^{\frac m2} G_{\frac k2 -1}(\frac{u \Psi^0(\sigma)}t)\ u^{k -1} du\ d\sigma
\\
& = \frac{\sigma_{k-1}}{\sigma_\Psi}
  (2\pi)^{-\frac k2}  \int_{\R^k} \int_0^\infty \left(\frac{1}{\cosh u}\right)^{\frac m2} G_{\frac k2 -1}(u \Psi^0(\sigma))\ u^{k -1} du\ d\sigma
\\
& = (2\pi)^{-\frac k2}   \int_{\R^k} \int_0^\infty \left(\frac{1}{\cosh u}\right)^{\frac m2} G_{\frac k2 -1}(u |\sigma|)\ u^{k -1} du\ d\sigma,
\end{align*}
where in the last equality we have used Lemma \ref{L:radial} again.
Comparing with \eqref{hankelFT} we see that
\begin{align}\label{last}
& \int_{\RN} \Gi dX =  (2\pi)^{-k}  \int_{\R^k} \hat h(\frac{\sigma}{2\pi}) d\sigma,
\end{align}
where 
\[
h(\la) = \left(\frac{1}{\cosh |\la|}\right)^{\frac m2}.
\]
Since $h\in \mathscr S(\R^k)$, and $h(0) = 1$, from the classical inversion formula for the Fourier transform we find
\[
\int_{\R^k} \hat h(\frac{\sigma}{2\pi}) d\sigma = (2\pi)^k \int_{\R^k} \hat h(\sigma) d\sigma = (2\pi)^k h(0) = (2\pi)^k.
\]
Substituting in \eqref{last}, we have proved \eqref{Gint}, thus completing the proof of Theorem \ref{T:main}.

\end{proof}

\medskip


\section{The nonlinear heat equation sees the anisotropic Minkowski gauge $\Theta^0$}\label{S:beau}

\medskip

This section is devoted to the  
\medskip

\begin{proof}[Proof of Theorem \ref{T:frompartoel}]
To save on space, we will compute the integral 
\begin{equation}\label{I}
\mathscr I = \int_0^\infty \frac{1}{t^{\frac{m}2 +k-1}} \int_0^\infty  \left(\frac{u}{\sinh u}\right)^{\frac m2} e^{-\frac{u}{\tanh u} \frac{\Phi^0(z)^2}{4t}} G_{\frac k2 -1}(\frac{u \Psi^0(\sigma)}t)\ u^{k -1}\ du\ \frac{dt}t,
\end{equation}
and then multiply the result by the constant
\begin{equation}\label{co}
\frac{\sigma_{m-1}\sigma_{k-1}}{\sigma_\Phi\sigma_\Psi} (2\pi)^{-\frac k2} (4\pi)^{-\frac m2}.
\end{equation}
To begin, we note that an obvious change of variable shows that
\begin{align}\label{I2}
\mathscr I & = \int_0^\infty t^{\frac{m}2 +k-2} \int_0^\infty  \left(\frac{u}{\sinh u}\right)^{\frac m2} e^{- t \frac{u}{\tanh u} \frac{\Phi^0(z)^2}{4}} G_{\frac k2 -1}(u \Psi^0(\sigma) t)\ u^{k -1}\ du\ dt
\\
& = \Psi^0(\sigma)^{1-\frac k2}\int_0^\infty \left(\frac{u}{\sinh u}\right)^{\frac m2} \ u^{\frac k2} \int_0^\infty t^{\frac{m}2 +\frac{k}2-1}   e^{- t \frac{u}{\tanh u} \frac{\Phi^0(z)^2}{4}} J_{\frac k2 -1}(u \Psi^0(\sigma) t)\ dt\ du,
\notag
\end{align}
where the second inequality is justified by an exchange of the order of integration and by having used \eqref{Gnu}.

To unravel the inner integral in $t$ in the right-hand side of \eqref{I2}, we now use the following formula due to Gegenbauer, see (3) on p. 385 in \cite{Wa}: suppose that \begin{equation}\label{req}
\Re(\nu+\mu)>0,\ \Re(\alpha+i\beta)>0,\ \Re(\alpha-i\beta)>0.
\end{equation}
Then one has  
\begin{equation}\label{futrll}
\int_0^\infty t^{\mu-1} e^{-\alpha t} J_\nu(\beta t) dt = \frac{2^{-\nu} \beta^\nu \G(\nu+\mu)}{\G(\nu+1) (\alpha^2 + \beta^2)^{\frac{\nu+\mu}2}} F\left(\frac{\nu+\mu}{2},\frac{1-\mu+\nu}{2};\nu+1;\frac{\beta^2}{\alpha^2 + \beta^2}\right), 
\end{equation}
In \eqref{futrll} we have denoted by 
\begin{equation}\label{F}
F(\alpha_1,\alpha_2;\beta_1;z) = \frac{\G(\beta_1)}{\G(\alpha_1)\G(\alpha_2)} \sum_{k=0}^\infty \frac{\G(k + \alpha_1) \G(k+\alpha_2)}{\G(k+\beta_1) k!} z^k,
\end{equation}
Gauss' hypergeometric function. We now assume $z\in \Rm\setminus\{0\}$, $\sigma\in \R^k\setminus\{0\}$, and for $u\in (0,\infty)$ fixed, apply \eqref{futrll} with the choice of parameters
\[
\nu = \frac k2 -1,\ \ \ \mu = \frac m2 + \frac k2,\ \ \ \alpha = \frac{u}{\tanh u} \frac{\Phi^0(z)^2}{4},\ \ \ \beta = u \Psi^0(\sigma),
\]
which gives
\[
\nu + \mu = \frac m2 + k - 1,\ \ \ \ 1-\mu+\nu = -\frac m2.
\]
Notice that, since $m\ge 1, k\ge 1$, we have $\nu+\mu\ge \frac 12>0$, and since $z\not= 0$, all requirements in \eqref{req} are fulfilled. Since
\[
\alpha^2 + \beta^2 = \frac{u^2}{16 \tanh^2 u} (\Phi^0(z)^4 + 16 \Psi^0(\sigma)^2 \tanh^2 u),\ \ \ \frac{\beta^2}{\alpha^2 + \beta^2} = \frac{16 \Psi^0(\sigma)^2 \tanh^2 u}{\Phi^0(z)^4 + 16 \Psi(\sigma)^2 \tanh^2 u},
\]
we obtain from \eqref{futrll}
\begin{align}\label{futral}
& \int_0^\infty t^{\frac{m}2 +\frac{k}2-1}   e^{- t \frac{u}{\tanh u} \frac{\Phi^0(z)^2}{4}} J_{\frac k2 -1}(u \Psi^0(\sigma) t)\ dt
= \frac{2^{1-\frac k2} (u\Psi^0(\sigma))^{\frac k2 -1} \G(\frac m2 + k - 1)}{\G(\frac k2) (\Phi^0(z)^4 + 16 \Psi^0(\sigma)^2 \tanh^2 u)^{\frac{1}2(\frac m2 + k - 1)}} 
\\
& \times \left(\frac{16 \tanh^2 u}{u^2}\right)^{\frac{1}2(\frac m2 + k - 1)} F\left(\frac{1}{2}(\frac m2 + k - 1),-\frac m4;\frac k2;\frac{16 \Psi^0(\sigma)^2 \tanh^2 u}{\Phi^0(z)^4 + 16 \Psi^0(\sigma)^2 \tanh^2 u}\right)
\notag
\\
& = \frac{4^{\frac m2 + k - 1}2^{1-\frac k2} \Psi^0(\sigma)^{\frac k2 -1}  \G(\frac m2 + k - 1)}{\G(\frac k2) } \frac{u^{\frac k2 -1} (\tanh^2 u)^{\frac{1}2(\frac m2 + k - 1)}}{u^{\frac m2 + k - 1} (\Phi^0(z)^4 + 16 \Psi^0(\sigma)^2 \tanh^2 u)^{\frac{1}2(\frac m2 + k - 1)}}
\notag
\\
& \times F\left(\frac{1}{2}(\frac m2 + k - 1),-\frac m4;\frac k2;\frac{16 \Psi^0(\sigma)^2 \tanh^2 u}{\Phi^0(z)^4 + 16 \Psi^0(\sigma)^2 \tanh^2 u}\right).
\notag
\end{align}
Substituting \eqref{futral} in \eqref{I2}, we find
\begin{align}\label{I3}
\mathscr I & = \frac{4^{\frac m2 + k - 1}2^{1-\frac k2}\G(\frac m2 + k - 1)}{\G(\frac k2)} \int_0^\infty \left(\frac{1}{\sinh^2 u}\right)^{\frac m4} \frac{(\tanh^2 u)^{\frac{1}2(\frac m2 + k - 1)}}{(\Phi^0(z)^4 + 16 \Psi^0(\sigma)^2 \tanh^2 u)^{\frac{1}2(\frac m2 + k - 1)}}
\\
& \times   F\left(\frac{1}{2}(\frac m2 + k - 1),-\frac m4;\frac k2;\frac{16 \Psi^0(\sigma)^2 \tanh^2 u}{\Phi^0(z)^4 + 16 \Psi^0(\sigma)^2 \tanh^2 u}\right)\ du,
\notag
\end{align}
It might be helpful for the reader to notice that in the above computation a first little miracle has happened: the powers of $u$ have disappeared from the integrand. This crucial aspect hides an important geometric information. To proceed in the computation we somehow need to kill the enemy, i.e., the very unpleasant factor 
\[
\frac{1}{(\Phi^0(z)^4 + 16 \Psi^0(\sigma)^2 \tanh^2 u)^{\frac{1}2(\frac m2 + k - 1)}}
\]
in the integral \eqref{I3}. The appropriate tool for the job is the following Kummer's relation (one of many), which prescribes the change of the hypergeometric function $F$ under linear transformations (see  formula (3) on p. 105 in \cite{E}),
\begin{equation}\label{hyperG}
F(\alpha,\beta;\gamma;x) = (1-x)^{-\alpha} F\left(\alpha,\gamma - \beta;\gamma;\frac{x}{x-1}\right),\ \ \ \ \ \ x\not=1, \ |\arg(1-x)|<\pi. 
\end{equation}
Comparing \eqref{I3} with \eqref{hyperG}, it appears evident that we should apply the latter with the choice
\begin{align*}
& \frac{x}{x-1} = \frac{16 \Psi^0(\sigma)^2 \tanh^2 u}{\Phi^0(z)^4 + 16 \Psi^0(\sigma)^2 \tanh^2 u},\ \ \ \gamma = \frac k2,\ \ \ \gamma - \beta = -\frac m4,\ \ \ \alpha = \frac{1}{2}(\frac m2 + k - 1).
\end{align*}
This choice gives
\[
x = - \frac{16 \Psi^0(\sigma)^2 \tanh^2 u}{\Phi^0(z)^4},\ \ \ \ \ \ \beta = \frac k2 + \frac m4.
\]
What is crucial for us is that
\[
1-x = \frac{\Phi^0(z)^4+ 16 \Psi^0(\sigma)^2 \tanh^2 u}{\Phi^0(z)^4}\ \not=\ 0, 
\]
and therefore
\[
(1-x)^\alpha = \left(\frac{\Phi^0(z)^4+ 16 \Psi^0(\sigma)^2 \tanh^2 u}{\Phi^0(z)^4}\right)^{\frac{1}{2}(\frac m2 + k - 1)}.
\]
Applying \eqref{hyperG} we thus find
\begin{align}\label{ultrafcktrll}
& F\left(\frac{1}{2}(\frac m2 + k - 1),-\frac m4;\frac k2;\frac{16 \Psi^0(\sigma)^2 \tanh^2 u}{\Phi^0(z)^4 + 16 \Psi^0(\sigma)^2 \tanh^2 u}\right) 
\\
& = \left(\frac{\Phi^0(z)^4+ 16 \Psi^0(\sigma)^2 \tanh^2 u}{\Phi^0(z)^4}\right)^{\frac{1}{2}(\frac m2 + k - 1)}
\notag\\
&  \times F\left(\frac{1}{2}(\frac m2 + k - 1),\frac m4+ \frac k2;\frac k2;- \frac{16 \Psi^0(\sigma)^2 \tanh^2 u}{\Phi^0(z)^4}\right).
\notag
\end{align}
Substitution of \eqref{ultrafcktrll} into \eqref{I3} gives
\begin{align}\label{I4}
\mathscr I & = \frac{4^{\frac m2 + k - 1}2^{1-\frac k2}\G(\frac m2 + k - 1)}{\G(\frac k2)\Phi^0(z)^{2(\frac m2 + k - 1)}} \int_0^\infty \left(\frac{1}{\sinh^2 u}\right)^{\frac m4} (\tanh^2 u)^{\frac{1}2(\frac m2 + k - 1)}
\\
& \times   F\left(\frac{1}{2}(\frac m2 + k - 1),\frac m4+ \frac k2;\frac k2;- \frac{16 \Psi^0(\sigma)^2 \tanh^2 u}{\Phi^0(z)^4}\right)\ du.
\notag
\end{align}
Formula \eqref{I4} represents the second little miracle in the proof of Theorem \ref{T:frompartoel}. Our next objective is to finally compute in closed form the integral in the right-hand side of \eqref{I4}. With this in mind, we intend to make the crucial change of variable $y = \tanh^2 u$, which gives $dy = \frac{2\tanh u}{\cosh^2 u} du$.
To do this we rearrange part of the integrand in the following way
\begin{align*}
& \left(\frac{1}{\sinh^2 u}\right)^{\frac m4} (\tanh^2 u)^{\frac{1}2(\frac m2 + k - 1)} 
 = \frac 12 \left(\frac{1}{\cosh^2 u}\right)^{\frac m4-1} (\tanh^2 u)^{\frac{k-2}2}\ \frac{2\tanh u}{\cosh^2 u}.
\end{align*}
We need to be a bit careful here, and distinguish the case $k=1$ from $k\ge 2$, but the relevant details are left to the reader. Assuming that $k\ge 2$ we rewrite \eqref{I4} as follows.
\begin{align}\label{I5}
\mathscr I & = \frac{4^{\frac m2 + k - 1}2^{1-\frac k2}\G(\frac m2 + k - 1)}{2\ \G(\frac k2)\Phi^0(z)^{2(\frac m2 + k - 1)}} \int_0^\infty \left(1- \tanh^2 u\right)^{\frac m4-1} (\tanh^2 u)^{\frac{k-2}2}\ \frac{2\tanh u}{\cosh^2 u}
\\
& \times   F\left(\frac{1}{2}(\frac m2 + k - 1),\frac m4+ \frac k2;\frac k2;- \frac{16 \Psi^0(\sigma)^2 \tanh^2 u}{\Phi^0(z)^4}\right)\ du.
\notag
\end{align}
Performing the above stated change of variable in the integral in the right-hand side of \eqref{I5}, we finally reach the conclusion that
\begin{align}\label{I6}
\mathscr I & = \frac{4^{\frac m2 + k - 1}2^{1-\frac k2}\G(\frac m2 + k - 1)}{2\ \G(\frac k2)\Phi^0(z)^{2(\frac m2 + k - 1)}} 
\int_0^1 \left(1-y\right)^{\frac m4-1} y^{\frac{k}2-1}\ 
\\
& \times   F\left(\frac{1}{2}(\frac m2 + k - 1),\frac m4+ \frac k2;\frac k2;- \frac{16 \Psi^0(\sigma)^2 }{\Phi^0(z)^4}\ y\right)\ dy.
\notag
\end{align}
The form of the integral in the right-hand side of \eqref{I6} represents the third miracle in the proof of the theorem. To unravel it, we apply the following formula due to H. Bateman, see \cite[(2) on p. 78]{E}, which gives
\begin{equation}\label{hyperGint}
\int_0^1  (1-y)^{\gamma-c-1} y^{c-1} F(\alpha,\beta;c;a y) dy = \frac{\G(c)\G(\gamma - c)}{\G(\gamma)} F(\alpha,\beta;\gamma;a),
\end{equation}
provided that 
\[
\Re \gamma >\Re c>0,\ \ \ \ \  a\not=1,\ \ \ \ \  |\arg(1-a)|<\pi.
\] 
Applying \eqref{hyperGint} with 
\[
\gamma = \frac m4 + \frac k2,\ \ \  c = \frac k2,\ \ \  \alpha = \frac{1}{2}(\frac m2 + k - 1),\ \ \ \beta = \frac m4+ \frac k2,\ \ \ \ a = - \frac{16 \Psi^0(\sigma)^2 }{\Phi^0(z)^4},
\]
we finally obtain
\begin{align}\label{I6}
\mathscr I & = \frac{4^{\frac m2 + k - 1}2^{-\frac k2}\G(\frac m2 + k - 1)\G(\frac m4)}{\G(\frac m4 + \frac k2)\Phi^0(z)^{2(\frac m2 + k - 1)}} 
 \ F\left(\frac{1}{2}(\frac m2 + k - 1),\frac m4+ \frac k2;\frac m4+ \frac k2;- \frac{16 \Psi^0(\sigma)^2 }{\Phi^0(z)^4}\right).
\end{align}
If we now use the following elementary, yet important property, which can be directly verified from \eqref{F}, see also \cite[(4) on p. 101]{E},
\begin{equation}\label{fs6}
F(\alpha,\beta;\beta;-a)  = (1+a)^{-\alpha},
\end{equation}
we reach the remarkable conclusion that
\begin{equation}\label{finalfcktr}
F\left(\frac{1}{2}(\frac m2 + k - 1),\frac m4+ \frac k2;\frac m4+ \frac k2;- \frac{16 \Psi^0(\sigma)^2 }{\Phi^0(z)^4}\right) = \frac{\Phi^0(z)^{2(\frac m2 + k - 1)}}{(\Phi^0(z)^4 + 16 \Psi^0(\sigma)^2)^{\frac 12(\frac m2 + k - 1)}}.
\end{equation}
Substituting \eqref{finalfcktr} in \eqref{I6}, we have
\begin{equation}\label{I7}
\mathscr I = \frac{4^{\frac m2 + k - 1}2^{-\frac k2}\G(\frac m4) \G(\frac m2 + k - 1)}{\G(\frac m4 + \frac k2)}  (\Phi^0(z)^4 + 16 \Psi^0(\sigma)^2)^{-\frac 12(\frac m2 + k - 1)}.
\end{equation}
To more conveniently represent the constant in the right-hand side of \eqref{I7}, we now use \eqref{prod} 
in which we take $x = \frac 12(\frac m2 + k -1)$, to find
\[
\frac{\G(\frac m2 + k -1)}{\G(\frac m4 + \frac k2)} = \pi^{-1/2}  2^{\frac m2 + k -2} \G(\frac 12(\frac m2 + k -1)).
\]
Substituting in \eqref{I7}, and keeping \eqref{theta01} and \eqref{Q2} in mind, we obtain
\begin{equation}\label{I8}
\mathscr I = \pi^{-1/2}  2^{\frac 3{2}m +\frac{5}{2}k -4} \G(\frac m4) \G(\frac 12(\frac m2 + k -1))\  \Theta^0(z,\sigma)^{-(Q-2)}.
\end{equation}
If we now remember that we need to multiply the right-hand side of \eqref{I8} by the constant in \eqref{co}, we finally reach the conclusion that 
\begin{equation}\label{I9}
\int_0^\infty \Gi\ dt = \frac{\sigma_{m-1}\sigma_{k-1}}{\sigma_\Phi\sigma_\Psi} \frac{2^{\frac m2 +2k-4} \G(\frac m4) \G(\frac 12(\frac m2 + k -1))}{\pi^{\frac{m+k+1}2}} \ \Theta^0(z,\sigma)^{-(Q-2)}.
\end{equation}
In view of \eqref{C} in Proposition \ref{P:const}, the equation \eqref{I9} finally proves \eqref{beauty}. 

\end{proof}



\bibliographystyle{amsplain}

\end{document}